\documentclass[preprint,3p,12pt]{elsarticle}

\usepackage{hyperref}
\usepackage[latin1]{inputenc}

\journal{\phantom{oo}}

\usepackage[all]{xy}
\usepackage[centertags]{amsmath}
\usepackage{latexsym}
\usepackage{amsfonts}
\usepackage{amssymb}
\usepackage{amsthm}

\usepackage [dvips]{epsfig}
\usepackage{graphicx}
\usepackage{newlfont}
\usepackage{color}

\def \dw {\Delta_{\omega}}
\def \d1 {\Delta_{1}}

\newcommand{\R}{\ensuremath{\mathbb{R}}}
\newcommand{\N}{\ensuremath{\mathbb{N}}}

\newtheorem{theorem}{Theorem}
\newtheorem{remark}[theorem]{Remark}
\newtheorem{lemma}[theorem]{Lemma}
\newtheorem{proposition}[theorem]{Proposition}
\newtheorem{definition}[theorem]{Definition}

\begin{document}

\begin{frontmatter}


\title{On multiple $\Delta_{\omega}$-Appell polynomials}

\author[Cam1]{P. Njionou Sadjang\corref{cor1}}
\author[Cam2]{S. Mboutngam}

\fntext[k1]{Corresponding author: pnjionou@yahoo.fr.}
\address[Cam1]{Faculty of Industrial Engineering, University of Douala, Cameroon}
\address[Cam2]{Higher Teachers' Training College, University of Maroua, Maroua, Cameroon}

\begin{abstract}
In this paper, we give a generating function for Multiple Charlier polynomials and deduce several consequences for these polynomials as invertion formula, connection formula, addition formula and recurrences relations they satisfy. Next, we introduce the notion of multiple $\Delta_{\omega}$ polynomials  and prove several equivalent conditions for this class of polynomials. Also, we give a characterization theorem that if multiple $\Delta_1$ polynomials are also multiple orthogonal, then they are the multiple Charlier polynomials.
\end{abstract}


\end{frontmatter}


\section{Introduction}

\noindent A sequence  $\{P_n(x)\}_{n=0}^{\infty}$ of $\dw$-Appell polynomials (see \cite[Definition 1.2.]{cheikh}) is a sequence of polynomials satisfying the difference property
\begin{equation}
\dw P_{n+1}(x)=(n+1)P_n(x),\quad n\geq 0,
\end{equation}
where
\begin{equation*}
\dw f(x)=\dfrac{f(x+\omega)-f(x)}{\omega},\quad \omega\neq 0,
\end{equation*}
Several properties for these polynomials are well known. Among them, the most important characterizations may be the following equivalent conditions \cite{cheikh}.\\
\noindent {\bf Theorem A.}\emph{
Let $\{P_n(x)\}_{n=0}^{\infty}$ be a sequence of polynomials. Then the following are all equivalent:
\begin{enumerate}
   \item $\{P_n(x)\}_{n=0}^{\infty}$ is a $\dw$-Appell polynomials set.
   \item There exists a sequence $(a_k)_{k\geq 0}$; independant of $n$; $a_0\neq 0$; such that
   \[P_n(x)=\sum_{k=0}^{n}a_k\dfrac{n!}{(n-k)!}x^{(n-k,\omega)}.\]
   \item $\{P_n(x)\}_{n=0}^{\infty}$ is generated by
   \[ A(t)(1+\omega t)^{\frac{x}{\omega}}=\sum_{n=0}^{\infty}P_n(x)\dfrac{t^n}{n!},\]
   with
   \[A(t)=\sum_{k=0}^{\infty}a_kt^k,\quad a_0\neq 0.\]
\end{enumerate}
where $\{x^{(n,\omega)}\}_{n\geq 0}$ is the polynomial set defined by
\[x^{(n,\omega)}=\left\{\begin{array}{ll}
x(x-\omega)(x-2\omega)\cdots(x-(n-1)\omega),& n=1,2,\ldots\\
1& n=0
\end{array}\right..\]}

\noindent {Note that for $\omega=1$, we will adopt the classical notation $x^{(n,1)}=x^{(n)}$ and the operator $\Delta_1$ will be denoted by $\Delta$.}

\noindent Among different generalizations of orthogonal polynomials,  multiple orthogonal polynomials are one of the most important.  They satisfy orthogonality conditions with respect to $r\in\N$ measures $\mu_1,\cdots \mu_r$ \cite{aptekarev98,aptekarev2003,arvesu,arvesu2,beckermann,bender,coussement01,nikishin}. The nonnegative $r$ represents  the number of weights. The origin of multiple orthogonal polynomial set goes back to Angelesco's paper dealing with simultaneaous Padé approximation \cite{angelesco}, in particular in Hermite-Pad\'e approximation of a system of $r$ (Markov) functions (see for instance \cite{bruin85,bruin90,mahler}). Some multiple discrete orthogonal polynomials are studied in \cite{arvesu} including multiple Charlier, multiple Meixner, multiple Kravchuk and multiple Hahn polynomials.\\
%
Some examples of $\Delta_1$-Appell polynomial sequence are $\{x^{(n,1)}\}_{n=0}^{\infty}$ and the Charlier polynomials. In particular, Charlier polynomials form the unique sequence of $\Delta_1$-Appell polynomials that is at the same time orthogonal with respect to a discrete measure (see for instance \cite{carlitz}). 

In this paper, section {\bf 2} provides a generating functions for the multiple Charlier polynomials. Next, this generating function is used to establish the inversion formula, the connection formulas, the addition formulas, and some recurrence relations for the multiple Charlier polynomials. Note that one of these recurrence relation already appears in the literature as we mentioned below. In section {\bf 3}. we introduced the notion of multiple $\Delta_\omega$ Appell polynomial sequence and prove several characterizations of these sequences. It is seen that the multiple Charlier polynomials are and example of such sequences. In section {\bf 4}, we prove that the only multiple orthogonal polynomial sequence which is also multiple $\Delta_1$-Appell are Charlier polynomials. This provides a new characterization result for the multiple Charlier polynomials. Finally section {\bf 5} contains some general results. To the best of our knowledge, except the results of theorem \ref{theo1}, which is obtained here by using the generating functions, the rest of results are new.

\section{Generating function for the multiple Charlier polynomials and consequences}

\noindent In the case where
\[ \mu_i(x)=\sum_{k=0}^{\infty}W^{(a_i)}(k)\delta(x-k),\quad i=1,2,\ldots,r\]
where $a_i>0$ and
\[W^{(a_i)}(x)=\left\{\begin{array}{ll}
\dfrac{a_i^x}{\Gamma(x+1)},& \textrm{if}\;\; x=0,1,\ldots\\
0&\textrm{otherwise}
\end{array}\right.,\]
we have the multiple Charlier polynomials. It is proved in \cite{lee2008} that for the multiple Charlier polynomials $\left\{C_{\vec{n}}^{(\vec{a})}(x)\right\}_{|\vec{n}|=0}^{\infty}$ , the following difference rule applies
\begin{equation}\label{appell111}
\Delta_1 C_{\vec{n}}^{(\vec{a})}(x)=\sum_{i=1}^{r}n_iC_{\vec{n}-\vec{e}_i}^{(\vec{a})}(x).
\end{equation}
For $r=2$, the system is
$\left\{C_{n_1,n_2}^{(a_1,a_2)}(x)\right\}_{n_1+n_2=0}^{\infty}$ and \eqref{appell111} becomes
\begin{equation}\label{char-app}
\Delta C_{n_1,n_2}^{(a_1,a_2)}(x)=n_1 C_{n_1-1,n_2}^{(a_1,a_2)}(x)+n_2C_{n_1,n_2-1}^{(a_1,a_2)}(x).
\end{equation}
Note also that the multiple Charlier polynomials have the representation \cite{arvesu}
\begin{equation}\label{multi-charlier-2}
C_{n_1,n_2}^{(a_1,a_2)}(x)= \sum_{k=0}^{n_1}\sum_{\ell=0}^{n_2}\binom{n_1}{k}\binom{n_2}{\ell}(-a_1)^{n_1-k}(-a_2)^{n_2-\ell}x^{(k+\ell,1)}.
\end{equation}

\noindent  Before we state the main theorem of this section, we recall the following important lemma.
\begin{lemma}\emph{(See \cite[Lemma 10]{rainville})}\label{lemma-sum}
The following relations apply:
\begin{eqnarray}
   \sum_{n=0}^{\infty}\sum_{k=0}^{\infty}A(k,n)&=&\sum_{n=0}^{\infty}\sum_{k=0}^{n}A(k,n-k)\\
   \sum_{n=0}^{\infty}\sum_{k=0}^{n}B(k,n)&=&\sum_{n=0}^{\infty}\sum_{k=0}^{\infty}B(k,n+k).
\end{eqnarray}
\end{lemma}
\noindent Now we give a generating function for the multiple Charlier polynomials $\{C_{n_1,n_2}^{(a_1,a_2)}(x)\}_{n_1+n_2=0}^{\infty}$.

\begin{theorem}
The multiple Charlier polynomials $C_{n_1,n_2}^{(a_1,a_2)}(x)$ have the following generating function:
\begin{equation}\label{genfunction1}
e^{-(a_1t_1+a_2t_2)}\big(1+t_1+t_2\big)^x=\sum_{n_1=0}^{\infty}\sum_{n_2=0}^{\infty}C_{n_1,n_2}^{(a_1,a_2)}(x)\dfrac{t_1^{n_1}}{n_1!}\dfrac{t_2^{n_2}}{n_2!}.
\end{equation}
\end{theorem}

\begin{proof}
In this proof we will use several times Lemma \ref{lemma-sum}. From Equation \eqref{multi-charlier-2}, we have:
\begin{eqnarray*}
&&\sum_{n_1=0}^{\infty}\sum_{n_2=0}^{\infty}C_{n_1,n_2}^{(a_1,a_2)}(x)\dfrac{t_1^{n_1}}{n_1!}\dfrac{t_2^{n_2}}{n_2!}\\
&&\hspace*{.5cm}= \sum_{n_1=0}^{\infty}\sum_{n_2=0}^{\infty}\left( \sum_{k=0}^{n_1}\sum_{\ell=0}^{n_2}\binom{n_1}{k}\binom{n_2}{\ell}(-a_1)^{n_1-k}(-a_2)^{n_2-\ell}x^{(k+\ell,1)}  \right)\dfrac{t_1^{n_1}}{n_1!}\dfrac{t_2^{n_2}}{n_2!}\\
&&\hspace*{.5cm}= \sum_{n_1=0}^{\infty}\sum_{k=0}^{n_1}\left(\sum_{n_2=0}^{\infty} \sum_{\ell=0}^{n_2}\binom{n_2}{\ell}(-a_2)^{n_2-\ell}x^{(k+\ell,1)} \dfrac{t_2^{n_2}}{n_2!} \right)(-a_1)^{n_1-k}\binom{n_1}{k}\dfrac{t_1^{n_1}}{n_1!}\\
&&\hspace*{.5cm}= \sum_{n_1=0}^{\infty}\sum_{k=0}^{n_1}\left(\sum_{n_2=0}^{\infty} \sum_{\ell=0}^{\infty}\binom{n_2+\ell}{\ell}(-a_2)^{n_2}x^{(k+\ell,1)} \dfrac{t_2^{n_2+\ell}}{(n_2+\ell)!} \right)(-a_1)^{n_1-k}\binom{n_1}{k}\dfrac{t_1^{n_1}}{n_1!}\\
&&\hspace*{.5cm}= \sum_{n_1=0}^{\infty}\sum_{k=0}^{n_1}
\left(\sum_{\ell=0}^{\infty} \left[\sum_{n_2=0}^{\infty}\dfrac{(-a_2t_2)^{n_2}}{n_2!}\right]\dfrac{t_2^{\ell}}{\ell!} x^{(k+\ell,1)} \right)
(-a_1)^{n_1-k}\binom{n_1}{k}\dfrac{t_1^{n_1}}{n_1!}\\
&&\hspace*{.5cm}= e^{-a_2t_2}\sum_{n_1=0}^{\infty}\sum_{k=0}^{n_1}
\left(\sum_{\ell=0}^{\infty} \dfrac{t_2^{\ell}}{\ell!} x^{(k+\ell,1)} \right)
(-a_1)^{n_1-k}\binom{n_1}{k}\dfrac{t_1^{n_1}}{n_1!}\\
&&\hspace*{.5cm}= e^{-a_2t_2}\sum_{\ell=0}^{\infty}\left(\sum_{n_1=0}^{\infty}\sum_{k=0}^{n_1}
(-a_1)^{n_1-k}\binom{n_1}{k}\dfrac{t_1^{n_1}}{n_1!}\right) \dfrac{t_2^{\ell}}{\ell!} x^{(k+\ell,1)}\\
&&\hspace*{.5cm}= e^{-a_2t_2}\sum_{\ell=0}^{\infty}\left(\sum_{n_1=0}^{\infty}\sum_{k=0}^{\infty}
(-a_1)^{n_1}\binom{n_1+k}{k}\dfrac{t_1^{n_1+k}}{(n_1+k)!}\right) \dfrac{t_2^{\ell}}{\ell!} x^{(k+\ell,1)}\\
&&=\hspace*{.5cm}e^{-(a_1t_1+a_2t_2)}\sum_{k=0}^{\infty}\sum_{\ell=0}^{\infty}x^{(k+\ell,1)}\dfrac{t_1^{k}}{k!}\dfrac{t_2^{\ell}}{\ell!}=e^{-(a_1t_1+a_2t_2)}\big(1+t_1+t_2\big)^x.
\end{eqnarray*}
The theorem is then proved.
\end{proof}

\begin{theorem}
The multiple Charlier polynomials fulfill the following inversion formula
\begin{equation}
x^{(n_1+n_2,1)}=\sum_{k=0}^{n_1}\sum_{\ell=0}^{n_2}\binom{n_1}{k}\binom{n_2}{\ell}(a_1)^{n_1-k}(a_2)^{n_2-\ell}C_{k,\ell}^{(a_1,a_2)}(x).
\end{equation}
\end{theorem}

\begin{proof}
Using the generating function \eqref{genfunction1}, we have
\begin{eqnarray*}
&&\sum_{n_1=0}^{\infty}\sum_{n_2=0}^{\infty}x^{(n_1+n_2,1)}\dfrac{t_1^{n_1}}{n_1!}\dfrac{t_2^{n_2}}{n_2!}= (1+t_1+t_2)^x\\
&&\hspace*{1.5cm} =e^{(a_1t_1+a_2t_2)}e^{-(a_1t_1+a_2t_2)}\big(1+t_1+t_2\big)^x\\
&&\hspace*{1.5cm} =e^{(a_1t_1+a_2t_2)}\sum_{n_1=0}^{\infty}\sum_{n_2=0}^{\infty}C_{n_1,n_2}^{(a_1,a_2)}(x)\dfrac{t_1^{n_1}}{n_1!}\dfrac{t_2^{n_2}}{n_2!}\\
&&\hspace*{1.5cm}
=\left(\sum_{k=0}^{\infty}\sum_{\ell=0}^{\infty}\dfrac{(a_1t_1)^k}{k!}\dfrac{(a_2t_2)^\ell}{\ell!}
\right)
\left( \sum_{n_1=0}^{\infty}\sum_{n_2=0}^{\infty}C_{n_1,n_2}^{(a_1,a_2)}(x)\dfrac{t_1^{n_1}}{n_1!}\dfrac{t_2^{n_2}}{n_2!}  \right)\\
&&\hspace*{1.5cm} =\sum_{n_1=0}^{\infty}\sum_{n_2=0}^{\infty}\left\{\sum_{k=0}^{n_1}\sum_{\ell=0}^{n_2}\binom{n_1}{k}\binom{n_2}{\ell}(a_1)^k(a_2)^\ell C^{(a_1,a_2)}_{n_1-k,n_2-\ell}(x)\right\}\dfrac{t_1^{n_1}}{n_1!}\dfrac{t_2^{n_2}}{n_2!}.
\end{eqnarray*}
Collecting the coefficients of $\dfrac{t_1^{n_1}}{n_1!}\dfrac{t_2^{n_2}}{n_2!}$, we obtain finally that
\begin{eqnarray*}
x^{(n_1+n_2,1)}&=& \sum_{k=0}^{n_1}\sum_{\ell=0}^{n_2}\binom{n_1}{k}\binom{n_2}{\ell}(a_1)^k(a_2)^{\ell}C^{(a_1,a_2)}_{n_1-k,n_2-\ell}(x)\\
&=& \sum_{k=0}^{n_1}\sum_{\ell=0}^{n_2}\binom{n_1}{k}\binom{n_2}{\ell}(a_1)^{n_1-k})(a_2)^{n_2-\ell}C^{(a_1,a_2)}_{k,\ell}(x).
\end{eqnarray*}
This ends the proof of the theorem.
\end{proof}

\begin{theorem}
The multiple Charlier polynomials $C_{n_1,n_2}^{(a_1,a_2)}(x)$ and $C_{n_1,n_2}^{(b_1,b_2)}(x)$ fulfill the following connection formula
\begin{equation}
C_{n_1,n_2}^{(b_1,b_2)}(x)=\sum_{k=0}^{n_1}\sum_{\ell=0}^{n_2}\binom{n_1}{k}\binom{n_2}{\ell}(a_1-b_1)^k(a_2-b_2)^\ell C^{(a_1,a_2)}_{n_1-k,n_2-\ell}(x).
\end{equation}
\end{theorem}

\begin{proof}
Using again the generating function \eqref{genfunction1}, we have
\begin{eqnarray*}
&&\sum_{n_1=0}^{\infty}\sum_{n_2=0}^{\infty}C_{n_1,n_2}^{(b_1,b_2)}(x)\dfrac{t_1^{n_1}}{n_1!}\dfrac{t_2^{n_2}}{n_2!}=e^{-(b_1t_1+b_2t_2)}\big(1+t_1+t_2\big)^x\\
&&\hspace*{2cm} = e^{(a_1-b_1)t_1+(a_2-b_2)t_2)}e^{-(a_1t_1+a_2t_2)}\big(1+t_1+t_2\big)^x\\
&&\hspace*{2cm}= \left(\sum_{k=0}^{\infty}\sum_{\ell=0}^{\infty}\dfrac{((a_1-b_1)t_1)^k}{k!}\dfrac{((a_2-b_2)t_2)^\ell}{\ell!}    \right) \left( \sum_{n_1=0}^{\infty}\sum_{n_2=0}^{\infty}C_{n_1,n_2}^{(a_1,a_2)}(x)\dfrac{t_1^{n_1}}{n_1!}\dfrac{t_2^{n_2}}{n_2!}  \right)\\
&&\hspace*{2cm} =\sum_{n_1=0}^{\infty}\sum_{n_2=0}^{\infty}\left\{\sum_{k=0}^{n_1}\sum_{\ell=0}^{n_2}\binom{n_1}{k}\binom{n_2}{\ell}(a_1-b_1)^k(a_2-b_2)^\ell C^{(a_1,a_2)}_{n_1-k,n_2-\ell}(x)\right\}\dfrac{t_1^{n_1}}{n_1!}\dfrac{t_2^{n_2}}{n_2!}.
\end{eqnarray*}
Collecting the coefficients of $\dfrac{t_1^{n_1}}{n_1!}\dfrac{t_2^{n_2}}{n_2!}$, the proof follows.
\end{proof}

\textcolor{black}{
\begin{theorem}
The following addition formula is valid
\begin{equation}
C_{n_1,n_2}^{(a,b)}(x+y)= \sum_{k=0}^{n_1}\sum_{\ell=0}^{n_2}\binom{n_1}{k}\binom{n_2}{\ell} C^{(a_1,b_1)}_{k,\ell}(x)C^{(a_2,b_2)}_{n_1-k,n_2-\ell}(y), 
\end{equation}
where $a_1+a_2=a$ and $b_1+b_2=b$.
\end{theorem}
\begin{proof}
Assuming that $a_1+a_2=a$ and $b_1+b_2=b$ and using the generating function \eqref{genfunction1}, we have
\begin{eqnarray*}
&&\sum_{n_1=0}^{\infty}\sum_{n_2=0}^{\infty}C_{n_1,n_2}^{(a,b)}(x+y)\dfrac{t_1^{n_1}}{n_1!}\dfrac{t_2^{n_2}}{n_2!}=e^{-(at_1+bt_2)}\big(1+t_1+t_2\big)^{x+y}\\
&&\hspace*{2cm} = e^{-(a_1t_1+b_1t_2)}\big(1+t_1+t_2\big)^{x}e^{-(a_2t_1+b_2t_2)}\big(1+t_1+t_2\big)^{y}\\
&&\hspace*{2cm}= \left(\sum_{k=0}^{\infty}\sum_{\ell=0}^{\infty}  C_{n_1,n_2}^{(a_1,b_1)}(x)\dfrac{t_1^{n_1}}{n_1!}\dfrac{t_2^{n_2}}{n_2!}    \right) \left( \sum_{n_1=0}^{\infty}\sum_{n_2=0}^{\infty}C_{n_1,n_2}^{(a_2,b_2)}(y)\dfrac{t_1^{n_1}}{n_1!}\dfrac{t_2^{n_2}}{n_2!}  \right)\\
&&\hspace*{2cm} =\sum_{n_1=0}^{\infty}\sum_{n_2=0}^{\infty}\left\{\sum_{k=0}^{n_1}\sum_{\ell=0}^{n_2}\binom{n_1}{k}\binom{n_2}{\ell} C^{(a_1,b_1)}_{k,\ell}(x)C^{(a_2,b_2)}_{n_1-k,n_2-\ell}(y)\right\}\dfrac{t_1^{n_1}}{n_1!}\dfrac{t_2^{n_2}}{n_2!}.
\end{eqnarray*}
Collecting the coefficients of $\dfrac{t_1^{n_1}}{n_1!}\dfrac{t_2^{n_2}}{n_2!}$, the proof follows.
\end{proof}}

\noindent Next, we prove some recurrence relations for the multiple Charlier polynomials $C_{n_1,n_2}^{(a_1,a_2)}(x)$.

\begin{proposition}
Let $G(x,t_1,t_2)=e^{-(a_1t_1+a_2t_2)}\big(1+t_1+t_2\big)^x$ be the generating function of the multiple Charlier polynomials $C_{n_1,n_2}^{(a_1,a_2)}(x)$. Then, the following equations are valid.
\begin{eqnarray}
   \dfrac{\partial G}{\partial t_1}(x,t_1,t_2)-\dfrac{\partial G}{\partial t_2}(x,t_1,t_2)&=& (a_2-a_1)G(x,t_1,t_2),\label{gen-rel1}\\
   \big(1+t_1+t_2\big)\dfrac{\partial G}{\partial t_1}(x,t_1,t_2)&=&\left(x-a_1\big(1+t_1+t_2\big)\right)G(x,t_1,t_2).\label{gen-rel2}
\end{eqnarray}
\end{proposition}

\begin{proof}
The proof follows from straighforward computations.
\end{proof}

\begin{theorem}\label{theo1}
The multiple Charlier polynomials $C_{n_1,n_2}^{(a_1,a_2)}(x)$
fulfill the following recurrence relations
\begin{eqnarray}
 (a_2-a_1)C_{n_1,n_2}^{(a_1,a_2)}(x)&=& C_{n_1+1,n_2}^{(a_1,a_2)}(x)-C_{n_1,n_2+1}^{(a_1,a_2)}(x),\label{rec-1}\\
  xC_{n_1,n_2}^{(a_1,a_2)}(x)&=& C_{n_1+1,n_2}^{(a_1,a_2)}(x)+(a_1+n_1+n_2)C_{n_1,n_2}^{(a_1,a_2)}(x)\nonumber\\
  &&+(a_1n_1+a_2n_2)C_{n_1,n_2-1}^{(a_1,a_2)}(x)+a_1n_1(a_1-a_2)C_{n_1-1,n_2-1}^{(a_1,a_2)}(x);\label{rec-2}\\
 \textcolor{black}{ xC_{n_1,n_2}^{(a_1,a_2)}(x)}&\textcolor{black}{=}&\textcolor{black}{ C_{n_1+1,n_2}^{(a_1,a_2)}(x)+(a_1+n_1+n_2)C_{n_1,n_2}^{(a_1,a_2)}(x)}\nonumber\\
  &&\hspace*{3cm}\textcolor{black}{ +a_1n_1C_{n_1-1,n_2}^{(a_1,a_2)}(x)+a_2n_2C_{n_1,n_2-1}^{(a_1,a_2)}(x).\label{rec-3}}
\end{eqnarray}
\end{theorem}

\begin{proof}
From \eqref{gen-rel1}, we have
\begin{eqnarray*}
(a_2-a_1)G(x,t_1,t_2)&=&\dfrac{\partial G}{\partial t_1}(x,t_1,t_2)-\dfrac{\partial G}{\partial t_2}(x,t_1,t_2).
\end{eqnarray*}
Hence,
\begin{eqnarray*}
\sum_{n_1=0}^{\infty}\sum_{n_2=0}^{\infty}(a_2-a_1)C_{n_1,n_2}^{(a_1,a_2)}(x)\dfrac{t_1^{n_1}}{n_1!}\dfrac{t_2^{n_2}}{n_2!}&=&\sum_{n_2=0}^{\infty}\left(C_{n_1+1,n_2}^{(a_1,a_2)}(x)-C_{n_1,n_2+1}^{(a_1,a_2)}(x)\right)\dfrac{t_1^{n_1}}{n_1!}\dfrac{t_2^{n_2}}{n_2!}.
\end{eqnarray*}
Collecting the coefficients of $\dfrac{t_1^{n_1}}{n_1!}\dfrac{t_2^{n_2}}{n_2!}$, \eqref{rec-1} follows. \\
In order to prove \eqref{rec-2}, we use \eqref{gen-rel2} as follows. The left-hand side of \eqref{gen-rel2} gives
\begin{eqnarray*}
 \big(1+t_1+t_2\big)\dfrac{\partial G}{\partial t_1}&=& \big(1+t_1+t_2\big) \sum_{n_1=1}^{\infty}\sum_{n_2=0}^{\infty}C_{n_1,n_2}^{(a_1,a_2)}(x)\dfrac{t_1^{n_1-1}}{(n_1-1)!}\dfrac{t_2^{n_2}}{n_2!}\\
 &=& \sum_{n_1=0}^{\infty}\sum_{n_2=0}^{\infty}\Big(C_{n_1+1,n_2}^{(a_1,a_2)}(x)+n_1C_{n_1,n_2}^{(a_1,a_2)}(x)+n_2C_{n_1+1,n_2-1}^{(a_1,a_2)}(x)\Big)\dfrac{t_1^{n_1}}{n_1!}\dfrac{t_2^{n_2}}{n_2!}.
\end{eqnarray*}
Next, the right-hand side of \eqref{gen-rel2} gives
\begin{eqnarray*}
&&\Big(x-a_1\big(1+t_1+t_2\big)\Big)G= \Big(x-a_1\big(1+t_1+t_2\big)\Big)\sum_{n_1=0}^{\infty}\sum_{n_2=0}^{\infty}C_{n_1,n_2}^{(a_1,a_2)}(x)\dfrac{t_1^{n_1}}{n_1!}\dfrac{t_2^{n_2}}{n_2!}\noindent \\
&&\hspace*{2.cm}=  \sum_{n_1=0}^{\infty}\sum_{n_2=0}^{\infty}\left[(x-a_1)C_{n_1,n_2}^{(a_1,a_2)}(x)-a_1n_1 C_{n_1-1,n_2}^{(a_1,a_2)}(x)-a_1n_2C_{n_1,n_2-1}^{(a_1,a_2)}(x) \right]\dfrac{t_1^{n_1}}{n_1!}\dfrac{t_2^{n_2}}{n_2!}.
\end{eqnarray*}
Collecting the coefficients of $\dfrac{t_1^{n_1}}{n_1!}\dfrac{t_2^{n_2}}{n_2!}$, we get
\begin{eqnarray}
&&C_{n_1+1,n_2}^{(a_1,a_2)}(x)+n_1C_{n_1,n_2}^{(a_1,a_2)}(x)+n_2C_{n_1+1,n_2-1}^{(a_1,a_2)}(x)\nonumber\\
&&\hspace*{2cm}=(x-a_1)C_{n_1,n_2}^{(a_1,a_2)}(x)-a_1n_1 C_{n_1-1,n_2}^{(a_1,a_2)}(x)-a_1n_2C_{n_1,n_2-1}^{(a_1,a_2)}(x).\label{for-22}
\end{eqnarray}
From \eqref{rec-1}, replacing $n_2$ by $n_2-1$  and $n_1$ by
$n_1-1$, we get:
\begin{eqnarray*}
C_{n_1+1,n_2-1}^{(a_1,a_2)}(x)&=& C_{n_1,n_2}^{(a_1,a_2)}(x)-(a_2-a_1)C_{n_1,n_2-1}^{(a_1,a_2)}(x)\\
C_{n_1-1,n_2}^{(a_1,a_2)}(x)&=& C_{n_1,n_2-1}^{(a_1,a_2)}(x)-(a_2-a_1)C_{n_1-1,n_2-1}^{(a_1,a_2)}(x).
\end{eqnarray*}
Subtituting these relations in \eqref{for-22}, \eqref{rec-2} follows.\\
\textcolor{black}{For the relation \eqref{rec-3}, we use again \eqref{rec-1} this time with \eqref{rec-2}. Replacing $n_1$ by $n_1-1$ and $n_2$ by $n_2-1$ in \eqref{rec-1}, we obtain
\[(a_2-a_1)C_{n_1-1,n_2-1}^{(a_1,a_2)}(x)= C_{n_1,n_2-1}^{(a_1,a_2)}(x)-C_{n_1-1,n_2}^{(a_1,a_2)}(x).\]
Then, \eqref{rec-2} becomes
\begin{eqnarray*}
xC_{n_1,n_2}^{(a_1,a_2)}(x)&=& C_{n_1+1,n_2}^{(a_1,a_2)}(x)+(a_1+n_1+n_2)C_{n_1,n_2}^{(a_1,a_2)}(x)\\
  &&+(a_1n_1+a_2n_2)C_{n_1,n_2-1}^{(a_1,a_2)}(x)-a_1n_1\left( C_{n_1,n_2-1}^{(a_1,a_2)}(x)-C_{n_1-1,n_2}^{(a_1,a_2)}(x) \right)\\
  &=& C_{n_1+1,n_2}^{(a_1,a_2)}(x)+(a_1+n_1+n_2)C_{n_1,n_2}^{(a_1,a_2)}(x)\\
  &&\hspace*{3cm}+a_1n_1 C_{n_1-1,n_2}^{(a_1,a_2)}(x)+a_2n_2C_{n_1,n_2-1}^{(a_1,a_2)}(x)
\end{eqnarray*}
This ends the proof of the theorem.}
\end{proof}

\textcolor{black}{\begin{remark}
Note that the recurrence equations \eqref{rec-1}, \eqref{rec-2}, \eqref{rec-3} appear in \cite{haneczok,van} but are obtained using different approach.  Note also that when $n_2=0$, $C_{n_1,0}^{(a_1,a_2)}(x)=\tilde{C}_{n_1}^{(a_1)}(x)$ is the classical  monic Charlier polynomials, and \eqref{rec-3} is exactely the normalized recurrence relation (9.14.4) for monic Charlier polynomials given in \cite{KLS}.
\end{remark}
\noindent Next, we state without proof some general results about multiple Charlier polynomials. Note that the proofs are similar to the case $r=2$.}

\begin{theorem}\label{generalized-generating-function}
The multiple Charlier polynomials $C_{\vec{n}}^{(\vec{a})}(x)$ have the following generating function:
\begin{equation}
e^{-\sum\limits_{j=1}^{r}a_jt_j}\Big(1+\sum_{j=1}^{r}t_j\Big)^x=\sum_{n_1=0}^{\infty}\sum_{n_2=0}^{\infty}\cdots\sum_{n_r=0}^{\infty}C_{\vec{n}}^{(\vec{a})}(x)\dfrac{t_1^{n_1}}{n_1!}\dfrac{t_2^{n_2}}{n_2!}\cdots
\dfrac{t_r^{n_r}}{n_r!}.
\end{equation}
\end{theorem}

\noindent \textcolor{black}{ From Theorem \ref{generalized-generating-function} we deduce the following propositions.
\begin{proposition}
The multiple Charlier polynomials $C_{\vec{n}}^{(\vec{a})}(x)$ fulfill the following inversion formula
\begin{equation*}
x^{(|\vec{n}|)}=\sum_{k_r=0}^{n_r}\cdots\sum_{k_1=0}^{n_1}\binom{n_1}{k_1}\cdots\binom{n_r}{k_r}(a_1)^{n_1-k_1}\cdots(a_r)^{n_r-k_r}C_{\vec{n}}^{(|\vec{k}|)}(x).
\end{equation*}
\end{proposition}
\begin{proposition}
The multiple Charlier polynomials $C_{\vec{n}}^{(\vec{a})}(x)$ and $C_{\vec{n}}^{(\vec{b})}(x)$ fulfill the following connection formula
\begin{equation*}
C_{\vec{n}}^{(\vec{b})}(x)=\sum_{k_r=0}^{n_r}\cdots\sum_{k_1=0}^{n_1}\binom{n_1}{k_1}\cdots\binom{n_r}{k_r}(a_1-b_1)^{k_1}\cdots(a_r-b_r)^{k_r} C^{(\vec{a})}_{\vec{n}-\vec{k}}(x).
\end{equation*}
\end{proposition}
\begin{proposition}
The following addition formula is valid
\begin{equation*}
C_{\vec{n}}^{(\vec{a})}(x+y)= \sum_{k_1=0}^{n_1}\ldots \sum_{k_r=0}^{n_r}\binom{n_1}{k_1}\ldots\binom{n_r}{k_r} C^{\vec{\alpha}}_{\vec{k}}(x)C^{\vec{a}-\vec{\alpha}}_{\vec{n}-\vec{k}}(y).
\end{equation*}
\end{proposition}
}

\section{Multiple $\dw$-Appell polynomial sets}

\noindent In this section we introduce the notion of multiple $\dw$-Appell polynomial set and provide some of the main characterizations of such families.

\begin{definition}
A multiple polynomial set $\{P_{n_1,n_2}(x)\}_{n_1+n_2=0}^{\infty}$ is called $\dw$-Appell polynomial set if and only if
\begin{equation}
\dw P_{n_1,n_2}(x)=n_1 P_{n_1-1,n_2}(x)+n_2P_{n_1,n_2-1}(x).
\end{equation}
\end{definition}

\noindent In view of \eqref{char-app}, it follows that the multiple
Charlier polynomials
$\left\{C_{n_1,n_2}^{(a_1,a_2)}(x)\right\}_{n_1+n_2=0}^{\infty}$
form a $\Delta$-Appell polynomial set.

\noindent The following theorem gives some characterizations of $\dw$-Appell polynomial sets.

\begin{theorem}
Let $\{P_{n_1,n_2}(x)\}_{n_1+n_2=0}^{\infty}$ be a polynomial set. The following assertions are equivalent:
\begin{enumerate}
   \item[{\bf (a)}] $\{P_{n_1,n_2}(x)\}_{n_1+n_2=0}^{\infty}$ is a $\dw$-Appell polynomial set.
   \item[{\bf (b)}] $\{P_{n_1,n_2}(x)\}_{n_1+n_2=0}^{\infty}$ is generated by
   \begin{equation}\label{ap2}
      A(t_1,t_2)\big(1+\omega(t_1+t_2)\big)^{\frac{x}{\omega}}=\sum_{n_1=0}^{\infty}\sum_{n_2=0}^{\infty}P_{n_1,n_2}(x)
      \dfrac{t_1^{n_1}}{n_1!}\dfrac{t_2^{n_2}}{n_2!},
   \end{equation}
   where
   \begin{equation}\label{ap1}
   A(t_1,t_2)=\sum_{n_1=0}^{\infty}\sum_{n_2=0}^{\infty}a_{n_1,n_2} {t_1^{n_1} \over {n_1}!}{t_2^{n_2}\over {n_2}!}, \quad a_{0,0}\neq 0.
   \end{equation}
   \item[{\bf (c)}] There exists a sequence $\{a_{n_1,n_2}\}_{n_1,n_1=0}^{\infty}$ with $a_{0,0}\neq 0$ such that
   \begin{equation}
     P_{n_1,n_2}(x)=\sum_{k_1=0}^{n_1}\sum_{k_2=0}^{n_2}\binom{n_1}{k_1}\binom{n_2}{k_2}a_{n_1-k_1,n_2-k_2}x^{(k_1+k_2,\omega)}.
   \end{equation}
   \item[{\bf (d)}] The following addition formula holds true:
   \begin{equation}
   P_{n_1,n_2}(x+y)=\sum_{k_1=0}^{n_1}\sum_{k_2=0}^{n_2}\binom{n_1}{k_1}\binom{n_2}{k_2}P_{n_1-k_1,n_2-k_2}(x)y^{(k_1+k_2,\omega)}.
   \end{equation}
   \item[{\bf(e)}] There exists a sequence $\{a_{n_1,n_2}\}_{n_1,n_1=0}^{\infty}$ with $a_{0,0}\neq 0$ such that
   \begin{equation}
     P_{n_1,n_2}(x)=\sum_{k_1}^{n_1}\sum_{k_2}^{n_2}\binom{n_1}{k_1}\binom{n_2}{k_2}
     \dfrac{(n_1+n_2-k_1-k_2)!}{(n_1+n_2)!} a_{k_1,k_2}\left\{\dw^{k_1+k_2}x^{(n_1+n_2,\omega)}\right\}.
   \end{equation}
\end{enumerate}
\end{theorem}

\begin{proof}
{\bf Step 1:  {\bf (a)} $\Longleftrightarrow $ {\bf (b)}}

\noindent We first prove that the assertion {\bf(a)} is equivalent to the
assertion {\bf(b)}.
\begin{eqnarray*}
\sum_{n_1=0}^{\infty}\sum_{n_2=0}^{\infty}\dw P_{n_1,n_2}(x) \dfrac{t^{n_1}}{n_1!}\dfrac{t^{n_2}}{n_2!}&=&
\dw \left(\sum_{n_1=0}^{\infty}\sum_{n_2=0}^{\infty}P_{n_1,n_2}(x)\dfrac{t^{n_1}}{n_1!}\dfrac{t^{n_2}}{n_2!}\right)\\
&=& \dw\left(A(t_1,t_2)\big(1+\omega(t_1+t_2)\big)^{\frac{x}{\omega}}\right)\\
&=& (t_1+t_2)A(t_1,t_2)\big(1+\omega(t_1+t_2)\big)^{\frac{x}{\omega}}\\
&=& \sum_{n_1=0}^{\infty}\sum_{n_2=0}^{\infty}P_{n_1,n_2}(x)\dfrac{t^{n_1+1}}{n_1!}\dfrac{t^{n_2}}{n_2!}
+\sum_{n_1=0}^{\infty}\sum_{n_2=0}^{\infty}P_{n_1,n_2}(x)\dfrac{t^{n_1}}{n_1!}\dfrac{t^{n_2+1}}{n_2!}\\
&=&\sum_{n_1=0}^{\infty}\sum_{n_2=0}^{\infty}\left(n_1P_{n_1-1,n_2}(x)+n_2P_{n_1,n_2}(x)\right)\dfrac{t^{n_1}}{n_1!}\dfrac{t^{n_2}}{n_2!}.
\end{eqnarray*}
This proves that the second assertion implies the first. Now, let us assume that the first assertion holds and assume that
\begin{equation}\label{prov1}
\sum_{n_1=0}^{\infty}\sum_{n_2=0}^{\infty}P_{n_1,n_2}(x)\dfrac{t^{n_1}}{n_1!}\dfrac{t^{n_2}}{n_2}=A(x,t_1,t_2)\big(1+\omega(t_1+t_2)\big)^{\frac{x}{\omega}}.
\end{equation}
Then, we apply $\dw$ to both sides of \eqref{prov1}. For the right-hand side, we obtain
\begin{eqnarray*}
&&\dw \left(A(x,t_1,t_2)\big(1+\omega(t_1+t_2)\big)^{\frac{x}{\omega}}\right)\\
&&\hspace*{1cm} =A(x,t_1,t_2)\dw \big(1+\omega(t_1+t_2)\big)^{\frac{x}{\omega}}+\dw A(x,t_1,t_2)\big(1+\omega(t_1+t_2)\big)^{\frac{x+\omega}{\omega}}\\
&&\hspace*{1cm}=\left[(t_1+t_2)A(x,t_1,t_2)+(1+\omega(t_1+t_2))\dw A(x,t_1,t_2)       \right]\big(1+\omega(t_1+t_2)\big)^{\frac{x}{\omega}}.
\end{eqnarray*}
The left-hand side gives
\begin{eqnarray*}
\sum_{n_1=0}^{\infty}\sum_{n_2=0}^{\infty}\dw P_{n_1,n_2}(x)\dfrac{t^{n_1}}{n_1!}\dfrac{t^{n_2}}{n_2}&=&
\sum_{n_1=0}^{\infty}\sum_{n_2=0}^{\infty}(n_1P_{n_1-1,n_2}(x)+n_2P_{n_1,n_2-1})\dfrac{t^{n_1}}{n_1!}\dfrac{t^{n_2}}{n_2}\\
&=& (t_1+t_2)A(x,t_1,t_2)\big(1+\omega(t_1+t_2)\big)^{\frac{x}{\omega}}.
\end{eqnarray*}
It follows that
\[(1+\omega(t_1+t_2))\dw A(x,t_1,t_2)=0\quad \forall t_1,t_2, x.\]
So $\dw A(x,t_1,t_2)=0$ for each $\omega\neq 0$. Hence
$A(x,t_1,t_2)=A(t_1,t_2)$. \\

\noindent {\bf Step 2: {\bf (b)} $\Longleftrightarrow$ {\bf (c)}}
The series expansion of
$\big(1+\omega(t_1+t_2)\big)^{\frac{x}{\omega}}$ gives:
\[
\big(1+\omega(t_1+t_2)\big)^{\frac{x}{\omega}}=\sum_{k_1=0}^{\infty}\sum_{k_2=0}^{\infty}x^{(k_1+k_2,\omega)}\dfrac{t^{k_1}}{k_1!}\dfrac{t^{k_2}}{k_2!}.
\]
Using the expression of $A(t_1,t_2)$ given by (\ref{ap1}), we have
\begin{eqnarray*}
  A(t_1,t_2)\big(1+\omega(t_1+t_2)\big)^{\frac{x}{\omega}} &=&
  \left( \sum_{n_1=0}^{\infty}\sum_{n_2=0}^{\infty}a_{n_1,n_2} {t_1^{n_1} \over {n_1}!}{t_2^{n_2}\over
  {n_2}!}\right)\left(\sum_{k_1=0}^{\infty}\sum_{k_2=0}^{\infty}x^{(k_1+k_2,\omega)}\dfrac{t^{k_1}}{k_1!}\dfrac{t^{k_2}}{k_2!} \right)  \\
  &=&\sum_{n_1=0}^{\infty}\sum_{n_2=0}^{\infty}\left\{ \sum_{k_1=0}^{n_1}\sum_{k_2=0}^{n_2}
  {a_{n_1-k_1,n_2-k_2} \over (n_1-k_1)!(n_2-k_2)!}
  {x^{(k_1+k_2,\omega)}\over k_1!k_2!}\right\}t_1^{n_1}t_2^{n_2}\\
  &=&\sum_{n_1=0}^{\infty}\sum_{n_2=0}^{\infty}\left\{
  \sum_{k_1=0}^{n_1}\sum_{k_2=0}^{n_2}
\binom{n_1}{k_1}\binom{n_2}{k_2}a_{n_1-k_1,n_2-k_2}x^{(k_1+k_2,\omega)}
  \right\}{t_1^{n_1} \over {n_1}!}{t_2^{n_2}\over
  {n_2}!}.
\end{eqnarray*}
 Using (\ref{ap2}), then, this proves that {\bf (b)} implies {\bf
 (c)}.

\noindent  The generating function $G(x,t_1,t_2)$ of the sequence
 $\{P_{n_1,n_2}(x)\}$ reads
 \begin{eqnarray*}
   G(x,t_1,t_2) &=& \sum_{n_1=0}^{\infty}\sum_{n_2=0}^{\infty} P_{n_1,n_2}(x){t_1^{n_1} \over {n_1}!}{t_2^{n_2}\over
  {n_2}!} \\
    &=& \sum_{n_1=0}^{\infty}\sum_{n_2=0}^{\infty}\left\{
  \sum_{k_1=0}^{n_1}\sum_{k_2=0}^{n_2}
\binom{n_1}{k_1}\binom{n_2}{k_2}a_{n_1-k_1,n_2-k_2}x^{(k_1+k_2,\omega)}
  \right\}{t_1^{n_1} \over {n_1}!}{t_2^{n_2}\over
  {n_2}!} \\
    &=& \left( \sum_{n_1=0}^{\infty}\sum_{n_2=0}^{\infty}a_{n_1,n_2} {t_1^{n_1} \over {n_1}!}{t_2^{n_2}\over
  {n_2}!}\right)\left(\sum_{k_1=0}^{\infty}\sum_{k_2=0}^{\infty}x^{(k_1+k_2,\omega)}\dfrac{t^{k_1}}{k_1!}\dfrac{t^{k_2}}{k_2!} \right)  \\
    &=& \left( \sum_{n_1=0}^{\infty}\sum_{n_2=0}^{\infty}a_{n_1,n_2} {t_1^{n_1} \over {n_1}!}{t_2^{n_2}\over
  {n_2}!}\right)\big(1+\omega(t_1+t_2)\big)^{\frac{x}{\omega}}.
 \end{eqnarray*}
 Then
 \[
G(x,t_1,t_2)=A(t_1,t_2)\big(1+\omega(t_1+t_2)\big)^{\frac{x}{\omega}},
 \]
 where $A(t_1,t_2)$ is given by (\ref{ap1}).\\

\noindent {\bf Step 3: {\bf (b)} $\Longleftrightarrow$ {\bf (d)}}

 \noindent
We prove that the assertion {\bf (b)} is equivalent to the assertion
{\bf (d)}. Using Lemma \ref{lemma-sum}, and the assertion {\bf (b)},
we have
\begin{eqnarray*}
\sum_{n_1=0}^{\infty}\sum_{n_2=0}^{\infty}P_{n_1,n_2}(x+y) \dfrac{t^{n_1}}{n_1!}\dfrac{t^{n_2}}{n_2}&=&
 A(t_1,t_2)\left( 1+\omega(t_1+t_2)\right)^{\frac{x+y}{\omega}}\\
&=& \left(\sum_{n_1=0}^{\infty}\sum_{n_2=0}^{\infty}P_{n_1,n_2}(x) \dfrac{t^{n_1}}{n_1!}\dfrac{t^{n_2}}{n_2}\right)\left(\sum_{k_1=0}^{\infty}\dfrac{y^{(k_1,\omega)}}{k_1!}(t_1+t_2)^{k_1}\right)\\
&=& \left(\sum_{n_1=0}^{\infty}\sum_{n_2=0}^{\infty}P_{n_1,n_2}(x) \dfrac{t^{n_1}}{n_1!}\dfrac{t^{n_2}}{n_2}\right)\left(\sum_{k_1=0}^{\infty}\sum_{j=0}^{k_1}y^{(k_1,\omega)}\dfrac{t_1^{j}}{j!}\dfrac{t_2^{k_1-j}}{(k_1-j)!}\right)\\
&=& \left(\sum_{n_1=0}^{\infty}\sum_{n_2=0}^{\infty}P_{n_1,n_2}(x) \dfrac{t^{n_1}}{n_1!}\dfrac{t^{n_2}}{n_2}\right)\left(\sum_{k_1=0}^{\infty}\sum_{k_2=0}^{\infty}y^{(k_1+k_2,\omega)}\dfrac{t_1^{k_1}}{k_1!}\dfrac{t_2^{k_2}}{k_2!}\right)\\
\end{eqnarray*}
\begin{eqnarray*}
\phantom{aaaaaaaaaaaaaaaa}&=&
\sum_{n_1=0}^{\infty}\sum_{n_2=0}^{\infty}\left\{\sum_{k_1=0}^{n_1}\sum_{k_2=0}^{n_2}\binom{n_1}{k_1}\binom{n_2}{k_2}
P_{n_1-k_1,n_2-k_2}(x)y^{(k_1+k_2,\omega)}\right\}\dfrac{t_1^{n_1}}{n_1!}\dfrac{t_2^{n_2}}{n_2!}.\\
\end{eqnarray*}

\noindent {\bf Step 4: {\bf (c)} $\Longleftrightarrow$ {\bf (e)}}
It can be easily proved by induction that
\begin{eqnarray*}
\dw^{k_1+k_2}x^{(n_1+n_2,\omega)}&=&
{(n_1+n_2)!\over (n_1+n_2-k_1-k_2)!}x^{(n_1+n_2-k_1-k_2,\omega)}.
\end{eqnarray*}
Hence,
\begin{eqnarray*}
  P_{n_1,n_2}(x)&=&\sum_{k_1=0}^{n_1}\sum_{k_2=0}^{n_2}\binom{n_1}{k_1}\binom{n_2}{k_2}a_{n_1-k_1,n_2-k_2}x^{(k_1+k_2,\omega)} \\
  &=& \sum_{k_1=0}^{n_1}\sum_{k_2=0}^{n_2}\binom{n_1}{n_1-k_1}\binom{n_2}{n_2-k_2}a_{k_1,k_2}x^{(n_1+n_2-k_1-k_2,\omega)} \\
   &=& \sum_{k_1=0}^{n_1}\sum_{k_2=0}^{n_2}\binom{n_1}{n_1-k_1}\binom{n_2}{n_2-k_2}a_{k_1,k_2}
   {(n_1+n_2-k_1-k_2)!\over
   (n_1+n_2)!}\dw^{k_1+k_2}x^{(n_1+n_2,\omega)}\\
   &=& \sum_{k_1=0}^{n_1}\sum_{k_2=0}^{n_2}\binom{n_1}{k_1}\binom{n_2}{k_2}a_{k_1,k_2}
   {(n_1+n_2-k_1-k_2)!\over
   (n_1+n_2)!}\dw^{k_1+k_2}x^{(n_1+n_2,\omega)}.
\end{eqnarray*}
 Thus the
announced equivalence is proved.
\end{proof}

\section{Multiple $\Delta$-Appell orthogonal polynomials}

\noindent In this section, we characterize those multiple orthogonal
polynomials sequences which are also multiple $\Delta$-Appell
sequences. Note that when $\omega=1$ we simply write $\Delta$
instead of $\Delta_1$.  Note also that the following product rule is
valid for the operator $\Delta$
\begin{equation}\label{e0}\Delta[f(x)g(x)]=f(x+1)\Delta g(x)+g(x)\Delta f(x).\end{equation} We
recall that every multiple orthogonal polynomial sequence
$\{P_{\vec{n}}(x)\}_{|\vec{n}|=0}^{\infty}$ satisfies a recurrence
relation of the form \cite{haneczok,ismail2009}
\begin{equation}\label{rtrr}
xP_{\vec{n}}(x)=P_{\vec{n}+\vec{e}_k}(x)+b_{\vec{n},k}P_{\vec{n}}(x)+\sum_{j=1}^{r}a_{\vec{n},j}P_{\vec{n}-\vec{e}_j}(x).
\end{equation}

\noindent When $r=2$, \eqref{rtrr} takes the form
\begin{equation}\label{rtrr-2}
xP_{m,n}(x)=P_{m+1,n}(x)+E_{m,n}P_{m,n}(x)+F_{m,n}P_{m-1,n}(x)+G_{m,n}P_{m,n-1}(x).
\end{equation}
or
\begin{equation}\label{rtrr-3}
xP_{m,n}(x)=P_{m,n+1}(x)+\tilde{E}_{m,n}P_{m,n}(x)+\tilde{F}_{m,n}P_{m,n-1}(x)+\tilde{G}_{m,n}P_{m-1,n}(x).
\end{equation}
Both \eqref{rtrr-2} and \eqref{rtrr-3} are valid. In this section, we will use the form \eqref{rtrr-2}.

\textcolor{black}{
\begin{theorem}
The multiple orthogonal polynomial sequence $\{P_{m,n}(x)\}_{m,n=0}^{\infty}$ that are also multiple $\Delta$-Appell sequence are the multiple Charlier polynomials.
%
\end{theorem}
}
\begin{proof}
Applying the operator $\Delta$ to the left-hand side of (\ref{rtrr-2})
and using (\ref{e0}) for $f(x)=x$ and $g(x)=P_{m,n}(x)$, we have:
\begin{eqnarray}
  \Delta[xP_{m,n}(x)] &=& (x+1)\Delta P_{m,n}(x)+P_{m,n}(x) \nonumber\\
   &=&
   (x+1)\left(mP_{m-1,n}(x)+nP_{m,n-1}(x)\right)+P_{m,n}(x)\nonumber.
\end{eqnarray}
Using the fact that $\{P_{m,n}(x)\}_{m,n=0}^{\infty}$ is also a
multiple $\Delta$-Appell sequence, the previous equation gives:
\begin{eqnarray}
  \Delta[xP_{m,n}(x)]  &=& (m+1)P_{m,n}(x)+nP_{m+1,n-1}(x)+m\left(E_{m-1,n}+1\right)P_{m-1,n}(x)\nonumber\\
   &&+ n\left(E_{m,n-1}+1\right)P_{m,n-1}(x)
   +\left(mG_{m-1,n}+nF_{m,n-1}\right)P_{m-1,n-1}(x)\nonumber\\
   &&+mF_{m-1,n}P_{m-2,n}(x)+nG_{m,n-1}P_{m,n-2}(x).\label{e2}
\end{eqnarray}
Applying the operator $\Delta$ to the right-hand side of (\ref{rtrr-2})
gives:
\begin{eqnarray}
  \Delta[xP_{m,n}(x)]  &=& (m+1)P_{m,n}(x)+nP_{m+1,n-1}(x)
   +\left(mG_{m,n}+nF_{m,n}\right)P_{m-1,n-1}(x)\nonumber\\
   &&+mE_{m,n}P_{m-1,n}(x)+nE_{m,n}P_{m,n-1}(x)+(m-1)F_{m,n}P_{m-2,n}(x)\nonumber\\
   &&+(n-1)G_{m,n}P_{m,n-2}(x).\label{e3}
\end{eqnarray}
(\ref{e2}) and (\ref{e3}) give the following system
\[
\left\{
\begin{array}{l}
  mG_{m,n}+nF_{m,n}=mG_{m-1,n}+nF_{m,n-1}, \\
  E_{m,n}=E_{m-1,n}+1, \\
  E_{m,n}=E_{m,n-1}+1, \\
  (m-1)F_{m,n}=mF_{m-1,n}, \\
  (n-1)G_{m,n}=nG_{m,n-1}.
\end{array}
\right.
\]
The previous system gives:
\[
\left\{
\begin{array}{l}
  F_{m,n}=mF_{1,n}, \\
  G_{m,n}=nG_{m,1}, \\
  E_{m,n}=m+n+E_{0,0}, \\
  G_{m,1}+F_{1,n}=G_{m-1,1}+F_{1,n-1}.
\end{array}
\right.
\]
Note that 
\[G_{m,1}+F_{1,n}=G_{m-1,1}+F_{1,n-1}\implies G_{m,1}-G_{m-1,1}=F_{1,n-1}-F_{1,n}=\lambda,\]
where $\lambda$ is a constant to be determined. The previous equations imply
\[G_{m,1}=(m-1)\lambda+G_{1,1},\quad\textrm{and}\quad F_{1,n}=-(n-1)\lambda+F_{1,1}.\]
Hence
\[G_{m,1}+F_{1,n}=G_{1,1}+F_{1,1}=G_{1,1}+F_{1,1}+\lambda(m-n)\quad \forall m,n.\]
This shows that $\lambda=0$ and so 
\[F_{m,n}=mF_{1,1},\quad G_{m,n}=nG_{1,1},\quad E_{m,n}=m+n+E_{0,0},\quad E_{0,0},F_{1,1}, G_{1,1}\in\R.\]
The conclusion follows from the recurrence relation \eqref{rec-3} of the multiple Charlier polynomials.
\end{proof}

\section{General results}

\noindent More generaly, we have the following definition and results.

\begin{definition}
A multiple polynomial set $\{P_{\vec{n}}(x)\}_{|\vec{n}|=0}^{\infty}$ is called $\dw$-Appell polynomial set if and only if
\begin{equation}
\dw P_{\vec{n}}(x)=\sum_{j=1}^r n_j P_{\vec{n}-\vec{e}_j}(x).
\end{equation}
\end{definition}

\begin{theorem}
Let $\{P_{\vec{n}}(x)\}_{|\vec{n}|=0}^{\infty}$ be a polynomial set. The following assertions are equivalent:
\begin{enumerate}
   \item $\{P_{\vec{n}}(x)\}_{|\vec{n}|=0}^{\infty}$ is a $\dw$-Appell polynomial set.
   \item $\{P_{\vec{n}}(x)\}_{|\vec{n}|=0}^{\infty}$ is generated by
   \begin{equation}
      A(t_1,t_2,\ldots,t_n)\big(1+\omega(t_1+t_2+\cdots t_r)\big)^{\frac{x}{\omega}}=\sum_{n_1=0}^{\infty}\sum_{n_2=0}^{\infty}
      \cdots\sum_{n_r=0}^{\infty}P_{\vec{n}}(x)\dfrac{t_1^{n_1}t_2^{n_2}\cdots t_r^{n_r}}{{n}_1!n_2!\cdots n_r!},
   \end{equation}
   where
   \begin{equation}
   A(t_1,t_2,\ldots,t_r)=\sum_{n_1=0}^{\infty}\sum_{n_2=0}^{\infty}\cdots\sum_{n_r=0}^{\infty}a_{\vec{n}}
   \dfrac{t_1^{n_1}t_2^{n_2}\cdots t_r^{n_r}}{{n}_1!n_2!\cdots n_r!}
   \cdots t_r^{n_r}, \quad a_{\vec{0}}\neq {0}.
   \end{equation}
   \item There exists a sequence $\{a_{\vec{n}}\}_{|\vec{n}|=0}^{\infty}$ with $a_{\vec{0}}\neq 0$ such that
   \[
P_{\vec{n}}(x)=\sum_{k_1=0}^{n_1}\sum_{k_2=0}^{n_2}\cdots\sum_{k_r=0}^{n_r}\binom{n_1}{k_1}\cdots\binom{n_r}{k_r}a_{\vec{n}-\vec{k}}x^{(k_1+\cdots+k_r,\omega)}
   \]
   \item The following addition formulas holds true:
   \begin{equation}
   P_{\vec{n}}(x+y)=\sum_{k_1=0}^{n_1}\sum_{k_2=0}^{n_2}\cdots\sum_{k_r=0}^{n_r}\binom{n_1}{k_1}\ldots\binom{n_r}{k_r}P_{n_1-k_1,\ldots,n_r-k_r}(x)y^{(k_1+\ldots +k_r,\omega)}.
   \end{equation}
   \item There exists a sequence $\{a_{\vec{n}}\}_{|\vec{n}|=0}^{\infty}$ with $a_{\vec{0}}\neq 0$ such that
   \begin{eqnarray*}
     P_{\vec{n}}(x)=\sum_{k_1=0}^{n_1}\cdots\sum_{k_r=0}^{n_r}\binom{n_1}{k_1}\ldots\binom{n_r}{k_r}
     \dfrac{(|\vec{n}|-|\vec{k}|)!}{|\vec{n}|!} a_{\vec{k}}\left\{\dw^{|\vec{k}|}x^{(|\vec{n}|,\omega)}\right\}.
   \end{eqnarray*}
\end{enumerate}
\end{theorem}

\textcolor{black}{
\begin{theorem}
The multiple orthogonal polynomial sequence $\{P_{\vec{n}}(x)\}_{|\vec{n}|=0}^{\infty}$ that are also multiple $\Delta$-Appell sequence are the multiple Charlier polynomials $\{C_{\vec{n}}^{\vec{a}}(x)\}_{|\vec{n}|=0}^{\infty}$.
\end{theorem}
}

\end{document}